\documentclass{article}%
\usepackage{graphicx}
\usepackage[intlimits]{amsmath}
\usepackage{latexsym}
\usepackage{amsfonts}
\usepackage{amssymb}%
\makeatletter
\newfont{\footsc}{cmcsc10 at 8truept}
\newfont{\footbf}{cmbx10 at 8truept}
\newfont{\footrm}{cmr10 at 10truept}
\newtheorem{theorem}{Theorem}

\newtheorem{corollary}[theorem]{Corollary}

\newtheorem{definition}[theorem]{Definition}

\newtheorem{lemma}[theorem]{Lemma}

\newtheorem{proposition}[theorem]{Proposition}

\newenvironment{proof}[1][Proof]{\noindent{\textbf {#1}  }}  {\hfill$\Box$\bigskip}

\begin{document}

\title{Ordering uniform supertrees by their spectral radii}
\author{Xiying Yuan\thanks{Department of Mathematics, Shanghai University, Shanghai
200444, China; \textit{email: xiyingyuan2007@hotmail.com }} \thanks{Research
supported by National Science Foundation of China (No. 11101263), and by a
grant of "The First-class Discipline of Universities in Shanghai".}}
\maketitle

\begin{abstract}
A connected and acyclic hypergraph is called a supertree. In this paper we
mainly focus on the spectral radii of uniform supertrees. Li, Shao and Qi
determined the first two $k$-uniform supertrees with large spectral radii
among all the $k$-uniform supertrees on $n$ vertices [H. Li, J. Shao, L. Qi,
The extremal spectral radii of $k$-uniform supertrees, arXiv:1405.7257v1, May
2014]. By applying the operation of moving edges on hypergraphs and using the
weighted incidence matrix method we extend the above order to the fourth
$k$-uniform supertree.



\end{abstract}

\section{Introduction}

Let $G$ be an ordinary graph, and $A(G)$ be its adjacency matrix. Denote by
$\rho(G)$ the spectral radius of graph $G,$ i.e., the largest eigenvalue of
$A(G).$ As usual, denote by $S_{n}$, $P_{n}$ the star on $n$ vertices, the
path on $n$ vertices, respectively.

We will take some notation from \cite{HuQiShao-power} and \cite{LiShaoQi}. We
denote the set $\{1,2,\cdot\cdot\cdot,n\}$ by $[n].$ Hypergraph is a natural
generalization of an ordinary graph (see \cite{Berge}). A hypergraph
$\mathcal{H}=(V(\mathcal{H}),E(\mathcal{H}))$ on $n$ vertices is a set of
vertices say $V(\mathcal{H})=\{1,2,\cdot\cdot\cdot,n\}$ and a set of edges,
say $E(\mathcal{H})=\{e_{1},e_{2},\cdot\cdot\cdot,e_{m}\},$ where
$e_{i}=\{i_{1},i_{2},\cdots,i_{l}\},i_{j}\in\lbrack n],$ $j=1,2,\cdots,l.$ If
$|e_{i}|=k$ for any $i=1,2,\cdot\cdot\cdot,m,$ then $\mathcal{H}$ is called
$k$-uniform hypergraph. A vertex $v$ is said to be incident to an edge $e$ if
$v\in e.$ The degree $d(i)$ of vertex $i$ is defined as $d(i)=|\{e_{j}:i\in
e_{j}\in E(\mathcal{H})\}|.$ A vertex of degree one is called a pendent
vertex. For a $k$-uniform hypergraph $\mathcal{H},$ an edge $e\in
E(\mathcal{H})$ is called a pendent edge if $e$ contains exactly $k-1$ pendent vertices.

An order $k$ dimension $n$ tensor $\mathcal{A=}(\mathcal{A}_{i_{1}i_{2}\cdots
i_{k}})\in\mathbb{C}^{n\times n\times\cdots\times n}$ is a multidimensional
array with $n^{k}$ entries, where $i_{j}\in\lbrack n]$ \ for each
$j=1,2,\cdot\cdot\cdot,k.$ To study the properties of uniform hypergraph by
algebraic methods, adjacency matrix of an ordinary graph is naturally
generalized to adjacency tenor (it is called adjacency hypermatrix in
\cite{Cooper}) of a hypergraph (see \cite{Cooper} \cite{Qi2014}).

\begin{definition}
Let $\mathcal{H}=(V(\mathcal{H}),E(\mathcal{H}))$ be a $k$-uniform hypergraph
on $n$ vertices. The adjacency tensor of $\mathcal{H}$ is defined as the
$k$-$th$ order $n$-dimensional tensor $\mathcal{A(H)}$ whose $(i_{1}\cdots
i_{k})$-entry is:
\[
\mathcal{A(H)}_{i_{1}i_{2}\cdots i_{k}}=%
\begin{cases}
\frac{1}{(k-1)!} & \text{$\{i_{1},i_{2},\cdots,i_{k}\}\in E(\mathcal{H})$}\\
0 & \text{otherwise}.
\end{cases}
\]

\end{definition}

The following general product of tensors, is defined in \cite{Shao} by Shao,
which is a generalization of the matrix case.

\begin{definition}
Let $\mathcal{A}$ and $\mathcal{B}$ be order $m\geq2$ and $k\geq1$ dimension
$n$ tensors, respectively. The product $\mathcal{AB}$ is the following tensor
$\mathcal{C}$ of order $(m-1)(k-1)+1$ and dimension $n$ with entries:
\begin{equation}
\mathcal{C}_{i\alpha_{1}\cdots\alpha_{m-1}}=\sum_{i_{2},\cdots,i_{m}\in\lbrack
n]}\mathcal{A}_{ii_{2}\cdots i_{m}}\mathcal{B}_{i_{2}\alpha_{1}}%
\cdots\mathcal{B}_{i_{m}\alpha_{m-1}}. \label{1}%
\end{equation}
Where $i\in\lbrack n],\alpha_{1},\cdots,\alpha_{m-1}\in\lbrack n]\times
\cdots\times\lbrack n]$.
\end{definition}

Let $\mathcal{A}$ be an order $k$ dimension $n$ tensor, let $x=(x_{1}%
,\cdot\cdot\cdot,x_{n})^{T}\in\mathbb{C}^{n}$ be a column vector of dimension
$n$. Then by (1) $\mathcal{A}x$ is a vector in $\mathbb{C}^{n}$ whose $i$th
component is as the following%

\[
(\mathcal{A}x)_{i}=\sum_{i_{2},\cdots,i_{k}=1}^{n}\mathcal{A}_{ii_{2}\cdots
i_{k}}x_{i_{2}}\cdots x_{i_{k}}.
\]

Let $x^{[k]}=(x_{1}^{k},\cdots,x_{n}^{k})^{T}$. Then (see \cite{ChangPZ}
\cite{Qi2014}) a number $\lambda\in\mathbb{C}$ is called an eigenvalue of the
tensor $\mathcal{A}$ if there exists a nonzero vector $x\in\mathbb{C}^{n}$
satisfying the following eigenequations%

\[
\mathcal{A}x=\lambda x^{[k-1]},
\]
and in this case, $x$ is called an eigenvector of $\mathcal{A}$ corresponding
to eigenvalue $\lambda$.

Let $\mathcal{A}$ be a $k$th-order $n$-dimensional nonnegative tensor. The
\textit{spectral radius} of $\mathcal{A}$ is defined as
\[
\rho(\mathcal{A})=\max\{|\lambda|:\lambda\text{ is an eigenvalue of
}\mathcal{A}\}.
\]

In this paper we call $\rho(\mathcal{A(H)})$ the spectral radius of uniform
hypergraph $\mathcal{H},$ denoted by $\rho(\mathcal{H}).$ For more details on
the eigenvalues of a uniform hypergraph one can refer to \cite{Cooper}
\cite{HuHuangLingQi} and \cite{Nikiforov}.

In \cite{FGH}, the weak irreducibility of nonnegative tensors was defined. It
was proved in \cite{FGH} and \cite{Yang} that a $k$-uniform hypergraph
$\mathcal{H}$ is connected if and only if its adjacency tensor $\mathcal{A(H)}%
$ is weakly irreducible.

\begin{theorem}
\label{Perron} \cite{ChangPZ} If $\mathcal{A}$ is a nonnegative tensor, then
$\rho(\mathcal{A})$ is an eigenvalue with a nonnegative eigenvector $x$
corresponding to it. If furthermore $\mathcal{A}$ is weakly irreducible, then
$x$ is positive, and for any eigenvalue $\lambda$ with nonnegative
eigenvector, $\lambda=\rho(\mathcal{A})$. Moreover, the nonnegative
eigenvector is unique up to a constant multiple.
\end{theorem}

By Theorem \ref{Perron}, for a $k$th-order weakly irreducible nonnegative
tensor $\mathcal{A}$, it has a unique positive eigenvector $x$ corresponding
to $\rho(\mathcal{A})$ with $||x||_{k}=1$ and it is called the principal
eigenvector of $\mathcal{A}$ (\cite{LiShaoQi}).

\begin{definition}
\cite{LiShaoQi} \label{supertree}A supertree is a hypergraph which is both
connected and acyclic.
\end{definition}

A characterization of acyclic hypergraph has been given in Berge's textbook
\cite{Berge}, and we just state a version for uniform hypergraphs.

\begin{proposition}
\cite{Berge} If $\mathcal{H}$ is a connected $k$-uniform hypergraph with $n$
vertices and $m$ edges, then it is acyclic if and only if $m(k-1)=n-1.$
\end{proposition}

The concept of \textit{power hypergraphs} was introduced in
\cite{HuQiShao-power}. Let $G=(V,E)$ be an ordinary graph. For every $k\geq3$,
the $k$th power of $G$, $G^{k}:=(V^{k},E^{k})$ is defined as the $k$-uniform
hypergraph with the edge set
\[
E^{k}:=\{e\cup\{i_{e,1},\cdot\cdot\cdot,i_{e,k-2}\}\text{ }|\text{ }e\in E\}
\]
and the vertex set
\[
V^{k}:=V\cup(\cup_{e\in E}\{i_{e,1},\cdot\cdot\cdot,i_{e,k-2}\}).
\]
The $k$th power of an ordinary tree was called a $k$\textit{-uniform
hypertree} (\cite{HuQiShao-power} \cite{LiShaoQi}). The following observations
are clear. Any $k$-uniform hypertree\textit{ }is a supertree. A $k$-uniform
supertree $\mathcal{T}$ with at least two edges is a $k$-uniform hypertree if
and only if each edge of $\mathcal{T}$ contains at most two non-pendent vertices.

The $k$th power of $S_{n},$ denoted by $S_{n}^{k},$ is called
\textit{hyperstar} in \cite{HuQiShao-power}. Let $S(a,b)$ be the tree on
$a+b+2$ vertices obtained from an edge $e$ by attaching $a$ pendent edges to
one end vertex of $e$, and attaching $b$ pendent edges to the other end vertex
of $e$. Let $S^{k}(a,b)$ be the $k$th power of $S(a,b$).

In \cite{LiShaoQi}, it was proved that the hyperstar $S_{n^{\prime}}^{k}$
attains uniquely the maximum spectral radius among all $k$-uniform supertrees
on $n$ vertices, and $S^{k}(1,n^{\prime}-3)$ attains uniquely the second
largest spectral radius among all $k$-uniform supertrees on $n$ vertices
(where $n^{\prime}=$ $\frac{n-1}{k-1}+1$).

Suppose that $m=\frac{n-1}{k-1},$ now we introduce a special class of
supertrees with $m$ edges, which are not hypertrees. Let $1\leq t_{1}\leq
t_{2}\leq t_{3}$ be three integers such that $t_{1}+t_{2}+t_{3}=m-1.$ Denote
by $\mathcal{T(}t_{1},t_{2},t_{3}\mathcal{)}$ the $k$-uniform supertree
containing exactly three non-pendent vertices, say $u_{1},u_{2},u_{3},$
incident to one edge, and $d(u_{i})=t_{i}+1$ holding for each $i=1,2,3$.

In this paper, we will determine the third and the fourth $k$-uniform
supertree with the large spectral radii among all $k$-uniform supertrees on
$n$ vertices.

\begin{theorem}
\label{Main 1}Let $\mathcal{T}$ be a $k$-uniform supertree on $n$ vertices
(with $m=$ $n^{\prime}-1$ edges, where $n^{\prime}=\frac{n-1}{k-1}+1\geq5$).
Suppose that $\mathcal{T\notin}\{S_{n^{\prime}}^{k},S^{k}(1,n^{\prime}-3)\}.$
Then we have
\[
\rho(\mathcal{T})\leq\rho(S^{k}(2,n^{\prime}-4)),
\]
with equality holding if and if $\mathcal{T\cong}S^{k}(2,n^{\prime}-4).$
\end{theorem}

\begin{theorem}
\label{Main 2}Let $\mathcal{T}$ be a $k$-uniform supertree on $n$ vertices
(with $m=$ $n^{\prime}-1$ edges, where $n^{\prime}=$ $\frac{n-1}{k-1}+1\geq
5$). Suppose that $\mathcal{T\notin}\{S_{n^{\prime}}^{k},S^{k}(1,n^{\prime
}-3),S^{k}(2,n^{\prime}-4)\}.$ Then we have
\[
\rho(\mathcal{T})\leq\rho(\mathcal{T(}1,1,m-3\mathcal{)}),
\]
with equality holding if and if $\mathcal{T\cong T(}1,1,m-3\mathcal{)}.$
\end{theorem}

The operation of moving edges on hypergraphs introduced by Li, Shao and Qi
(\cite{LiShaoQi}) and the weighted incidence matrix method introduced by Lu
and Man (\cite{LuMan}) are crucial for our proofs. In Section 2 we will show
them and other useful tools. In Section 3, we will give the proofs of our main results.

\section{Several tools to compare spectral radii}

A novel method (we call it weighted incidence matrix method) for computing (or
comparing) the spectral radii of hypergraphs was raised by Lu and Man.

\begin{definition}
\label{weighted incidence matrix} \cite{LuMan} A weighted incidence matrix $B$
of a hypergraph $\mathcal{H}=(V,E)$ is a $|V|\times|E|$ matrix such that for
any vertex $v$ and any edge $e$, the entry $B(v,e)>0$ if $v\in e$ and
$B(v,e)=0$ if $v\notin e$.
\end{definition}

\begin{definition}
\cite{LuMan} A hypergraph $\mathcal{H}$ is called $\alpha$-normal if there
exists a weighted incidence matrix $B$ satisfying

(1). $%
%TCIMACRO{\dsum \limits_{e:v\in e}}%
%BeginExpansion
{\displaystyle\sum\limits_{e:v\in e}}
%EndExpansion
$ $B(v,e)=1$, for any $v\in V(\mathcal{H})$.

(2). $%
%TCIMACRO{\dprod \limits_{v:v\in e}}%
%BeginExpansion
{\displaystyle\prod\limits_{v:v\in e}}
%EndExpansion
$ $B(v,e)=\alpha$, for any $e\in E(\mathcal{H})$.

Moreover, the weighted incidence matrix $B$ is called consistent if for any
cycle $v_{0}e_{1}v_{1}e_{2}\cdot\cdot\cdot v_{l}(v_{l}=v_{0})$%
\[%
%TCIMACRO{\dprod \limits_{i=1}^{l}}%
%BeginExpansion
{\displaystyle\prod\limits_{i=1}^{l}}
%EndExpansion
\frac{B(v_{i},e_{i})}{B(v_{i-1},e_{i})}=1.
\]

\end{definition}

\begin{definition}
\label{subnormal} \cite{LuMan} A hypergraph $\mathcal{H}$ is called $\alpha
$-subnormal if there exists a weighted incidence matrix $B$ satisfying

(1). $%
%TCIMACRO{\dsum \limits_{e:v\in e}}%
%BeginExpansion
{\displaystyle\sum\limits_{e:v\in e}}
%EndExpansion
$ $B(v,e)\leq1$, for any $v\in V(\mathcal{H})$.

(2). $%
%TCIMACRO{\dprod \limits_{v:v\in e}}%
%BeginExpansion
{\displaystyle\prod\limits_{v:v\in e}}
%EndExpansion
$ $B(v,e)\geq\alpha$, for any $e\in E(\mathcal{H})$.

Moreover, $\mathcal{H}$ is called strictly $\alpha$-subnormal if it is
$\alpha$-subnormal but not $\alpha$-normal.
\end{definition}

\begin{definition}
\label{supernormal}\cite{LuMan} A hypergraph $\mathcal{H}$ is called $\alpha
$-supernormal if there exists a weighted incidence matrix $B$ satisfying

(1). $%
%TCIMACRO{\dsum \limits_{e:v\in e}}%
%BeginExpansion
{\displaystyle\sum\limits_{e:v\in e}}
%EndExpansion
$ $B(v,e)\geq1$, for any $v\in V(\mathcal{H})$.

(2). $%
%TCIMACRO{\dprod \limits_{v:v\in e}}%
%BeginExpansion
{\displaystyle\prod\limits_{v:v\in e}}
%EndExpansion
$ $B(v,e)\leq\alpha$, for any $e\in E(\mathcal{H})$.

Moreover, $\mathcal{H}$ is called strictly $\alpha$-supernormal if it is
$\alpha$-supernormal but not $\alpha$-normal.
\end{definition}

For a fixed $k$-uniform hypergraph $\mathcal{H}$, $\rho(\mathcal{H})$ defined
here times constant factor $(k-1)!$ is the value of $\rho(\mathcal{H})$
defined in \cite{LuMan}$.$ While this is not essential. Remembering this
difference we modify Lemma 3 and Lemma 4 of \cite{LuMan} as the following
Theorem \ref{LuMan}.

\begin{theorem}
\cite{LuMan} \label{LuMan} Let $\mathcal{H}$ be a $k$-uniform hypergraph.

(1). If $\mathcal{H}$ is strictly $\alpha$-subnormal, then we have
$\rho(\mathcal{H})<\alpha^{-\frac{1}{k}}.$

(2). If $\mathcal{H}$ is strictly and consistently $\alpha$-supernormal, then
$\rho(\mathcal{H})>\alpha^{-\frac{1}{k}}.$
\end{theorem}

The following result reveals the numerical relationship between $\rho(G^{k})$
and $\rho(G),$ where $G^{k}$ is the $k$-th power of an ordinary graph $G.$

\begin{theorem}
\cite{ZhouBu} \label{Bu} Let $G^{k}$ be the $k$th power of an ordinary graph
$G.$ Then we have
\[
\rho(G^{k})=(\rho(G))^{\frac{2}{k}.}%
\]

\end{theorem}

Let $F_{n}$ $(n\geq5)$ be the tree obtained by coalescing the center of the
star $S_{n-4}$ and the center of the path $P_{5}.$ Ordering the trees on $n$
vertices according to their spectral radii was well studied in \cite{Hofm},
\cite{ChangHuang} and \cite{LinGuo}. We outline parts of the work in
\cite{Hofm} as follows.

\begin{theorem}
\cite{Hofm} \label{Hofm}Let $T$ be a tree on $n$ vertices $(n\geq5)$ and
$T\notin\{S_{n},S(1,n-3),S(2,n-4),F_{n}\}.$ Then we have
\[
\rho(S_{n})>\rho(S(1,n-3))>\rho(S(2,n-4))>\rho(F_{n})>\rho(T).
\]

\end{theorem}

Combining Theorems \ref{Bu} and \ref{Hofm}, we have the following corollary.

\begin{corollary}
\label{Ordering Tk}Let $T^{k}$ be the $k$th power of an ordinary tree $T$.
Suppose that $T^{k}$ has $n$ vertices, and $n^{\prime}=$ $\frac{n-1}%
{k-1}+1\geq5.$ Suppose $T\notin\{S_{n^{\prime}},S(1,n^{\prime}%
-3),S(2,n^{\prime}-4),F_{n^{\prime}}\},$ then we have
\[
\rho(S_{n^{\prime}}^{k})>\rho(S^{k}(1,n^{\prime}-3))>\rho(S^{k}(2,n^{\prime
}-4))>\rho(F_{n^{\prime}}^{k})>\rho(T^{k}).
\]

\end{corollary}

\begin{definition}
\label{moving edges} \cite{LiShaoQi} Let $r\geq1$, $\mathcal{G}=(V,E)$ be a
hypergraph with $u\in V$ and $e_{1},\cdot\cdot\cdot,e_{r}$ $\in E$, such that
$u\notin e_{i}$ for $i=1,\cdot\cdot\cdot,r.$ Suppose that $v_{i}\in e_{i}$ and
write $e_{i}^{\prime}$ $=(e_{i}\backslash\{v_{i}\})$ $\cup\{u\}(i=1,\cdot
\cdot\cdot,r)$. Let $\mathcal{G}^{\prime}$ $=(V,E^{\prime})$ be the hypergraph
with $E^{\prime}=(E\backslash\{e_{1},\cdot\cdot\cdot,e_{r}\})\cup
\{e_{1}^{\prime},\cdot\cdot\cdot,e_{r}^{\prime}\}.$ Then we say that
$\mathcal{G}^{\prime}$ is obtained from $\mathcal{G}$ by moving edges
$(e_{1},\cdot\cdot\cdot,e_{r})$ from $(v_{1},\cdot\cdot\cdot,v_{r})$ to $u$.
\end{definition}

The effect on $\rho(\mathcal{G})$ of moving edges was studied by Li, Shao and
Qi (see Theorem \ref{moving edge}). The following fact was pointed out in
\cite{LiShaoQi}. If $\mathcal{G}$ is acyclic and there is an edge $e\in
E(\mathcal{G})$ containing all the vertices $u,v_{1},\cdot\cdot\cdot,v_{r},$
then the graph $\mathcal{G}^{\prime}$ defined as above contains no multiple edges.

\begin{theorem}
\label{moving edge} \cite{LiShaoQi} Let $r\geq1$, $\mathcal{G}$ be a connected
hypergraph, $\mathcal{G}^{\prime}$ be the hypergraph obtained from
$\mathcal{G}$ by moving edges $(e_{1},\cdot\cdot\cdot,e_{r})$ from
$(v_{1},\cdot\cdot\cdot,v_{r})$ to $u$, and $\mathcal{G}^{\prime}$ contain no
multiple edges. If $x$ is the principal eigenvector of $\mathcal{A}%
(\mathcal{G})$ corresponding to $\rho(\mathcal{G})$ and suppose that
$x_{u}\geq max_{1\leq i\leq r}\{x_{v_{i}}\}$, then $\rho(\mathcal{G}^{\prime
})>\rho(\mathcal{G}).$
\end{theorem}

Denote by $N_{2}(\mathcal{T})$ the number of non-pendent vertices of
$\mathcal{T}.$ By using Theorem \ref{moving edge} (or modifying parts of the
proof of Theorem 21 of \cite{LiShaoQi}), we have the following observation.

\begin{lemma}
\label{induction}Let $\mathcal{T}$ be a $k$-uniform supertree on $n$ vertices
with $N_{2}(\mathcal{T})\geq2.$ Then there exists a $k$-uniform supertree
$\mathcal{T}^{\prime}$ on $n$ vertices with $N_{2}(\mathcal{T}^{\prime}%
)=N_{2}(\mathcal{T})-1$ and $\rho(\mathcal{T}^{\prime})>\rho(\mathcal{T}).$
\end{lemma}

\begin{lemma}
\cite{LiShaoQi} \label{S(a,b)}Let $a,b,c,d$ be nonnegative integers with
$a+b=c+d.$ Suppose that $a\leq b,$ $c\leq d$ and $a<c$, then we have
$\rho(S^{k}(a,b))>\rho(S^{k}(c,d)).$
\end{lemma}

\begin{lemma}
\label{T(a,b,c)}Let $1\leq$ $t_{1}\leq t_{2}\leq t_{3}$ be three integers with
$t_{1}+t_{2}+t_{3}=m-1.$ Then we have
\[
\rho(\mathcal{T(}1,1,m-3\mathcal{)})\geq\rho(\mathcal{T(}t_{1},t_{2}%
,t_{3}\mathcal{)}),
\]
with equality holding if and only if $t_{2}=1.$
\end{lemma}

\begin{proof}
If $t_{2}=1,$ the result is obvious. Now we suppose $t_{2}>1,$ thus $t_{3}>1.$
Let $u_{1},u_{2}$ and $u_{3}$ be the (only) three non-pendent vertices of
$\mathcal{T(}t_{1},t_{2},t_{3}\mathcal{)}$ with $d(u_{i})=t_{i}+1$, $i=1,2,3.$
It is easy to see that $u_{i}$ is incident to $t_{i}$ pendent edges,
$i=1,2,3.$ Let $x$ be the principal eigenvector of $\mathcal{A}(\mathcal{T(}%
t_{1},t_{2},t_{3}\mathcal{)})$ corresponding to $\rho(\mathcal{T(}t_{1}%
,t_{2},t_{3}\mathcal{)}).$ Without loss of generality we suppose that
$x_{u_{3}}=max_{1\leq i\leq3}\{x_{u_{i}}\}$. Let $\mathcal{G}$ be obtained
from $\mathcal{T(}t_{1},t_{2},t_{3}\mathcal{)}$ by moving $t_{1}-1$ pendent
edges from $u_{1}$ to $u_{3},$ and moving $t_{2}-1$ pendent edges from $u_{2}$
to $u_{3}.$ Then $\mathcal{G}$ is isomorphic to $\mathcal{T(}%
1,1,m-3\mathcal{)}$. Noting that $t_{2}>1,$ by Theorem \ref{moving edge} we
have $\rho(\mathcal{T(}1,1,m-3\mathcal{)})>\rho(\mathcal{T(}t_{1},t_{2}%
,t_{3}\mathcal{)})$.
\end{proof}

By Theorem \ref{Bu} we know that $\rho(S^{k}(2,n^{\prime}-4))$ is determined
by $\rho(S(2,n^{\prime}-4)),$ and $\rho(F_{n^{\prime}}^{k})$ is determined
by\ $\rho(F_{n^{\prime}}).$ We will use the weighted incidence matrix method
to compare $\rho(\mathcal{T(}1,1,m-3\mathcal{)})$ with $\rho(S^{k}%
(2,n^{\prime}-4))$ and $\rho(F_{n^{\prime}}^{k}).$

\begin{lemma}
\label{T(1,,1,)}Suppose that $n^{\prime}=\frac{n-1}{k-1}+1,$ $m=n^{\prime
}-1\geq4.$ We have
\[
\rho(S^{k}(2,n^{\prime}-4))>\rho(\mathcal{T(}1,1,m-3\mathcal{)})>\rho
(F_{n^{\prime}}^{k}).
\]

\end{lemma}

\begin{proof}
Denote by $u_{1},u_{2}$ and $u_{3}$ three non-pendent vertices of
$\mathcal{T(}1,1,m-3\mathcal{)}$. Label the $m$ edges of $\mathcal{T(}%
1,1,m-3\mathcal{)}$ as follows. The unique non-pendent edge (the edge
containing $u_{1},u_{2}$ and $u_{3})$ is numbered $e_{0},$ the pendent edge
containing $u_{1}$ is numbered $e_{1},$ the pendent edge containing $u_{2}$ is
numbered $e_{2},$ and the pendent edges containing $u_{3}$ are numbered
$e_{3},\cdot\cdot\cdot,e_{m-1}.$ Now we construct an $n\times m$ matrix $B.$
For any vertex $v$ and any edge $e$ of $\mathcal{T(}1,1,m-3\mathcal{)}$, let
$B(v,e)=0$ if $v\notin e$. For any pendent vertex $v$ in an edge $e,$ let
$B(v,e)=1$. For the non-pendent vertices $u_{1},u_{2}$ and $u_{3},$ let
$B(u_{1},e_{1})=\alpha,B(u_{1},e_{0})=1-\alpha;$ $B(u_{2},e_{2})=\alpha
,B(u_{2},e_{0})=1-\alpha;$ and let $B(u_{3},e_{i})=\alpha,$ for $i=3,\cdot
\cdot\cdot,m-1$, $B(u_{3},e_{0})=1-(m-3)\alpha.$ According to the above rules,
we say that for any vertex $v$ of $\mathcal{T(}1,1,m-3\mathcal{)}$ we have
\begin{equation}%
%TCIMACRO{\dsum \limits_{e:v\in e}}%
%BeginExpansion
{\displaystyle\sum\limits_{e:v\in e}}
%EndExpansion
B(v,e)=1. \label{sum}%
\end{equation}
For the pendent edge $e_{i}$ $(i=1,2,\cdot\cdot\cdot,m-1),$ we have
\begin{equation}%
%TCIMACRO{\dprod \limits_{v:v\in e_{i}}}%
%BeginExpansion
{\displaystyle\prod\limits_{v:v\in e_{i}}}
%EndExpansion
B(v,e_{i})=\alpha. \label{P1}%
\end{equation}
For the unique non-pendent edge $e_{0}$ we have
\[%
%TCIMACRO{\dprod \limits_{v:v\in e_{0}}}%
%BeginExpansion
{\displaystyle\prod\limits_{v:v\in e_{0}}}
%EndExpansion
B(v,e_{0})=(1-\alpha)^{2}[1-(m-3)\alpha],
\]
and then
\begin{equation}%
%TCIMACRO{\dprod \limits_{v:v\in e_{0}}}%
%BeginExpansion
{\displaystyle\prod\limits_{v:v\in e_{0}}}
%EndExpansion
B(v,e_{0})-\alpha=-(m-3)\alpha^{3}+(2m-5)\alpha^{2}-m\alpha+1. \label{alapha}%
\end{equation}
(1). Write $\rho=\rho(S(2,n^{\prime}-4))$ for short. By Theorem \ref{Bu}, we
have $\rho(S^{k}(2,n^{\prime}-4))=\rho^{\frac{2}{k}}.$ It is easy to check
that the tree $S(2,n^{\prime}-4)$ contains $m$ edges and the value $\rho$
satisfies
\begin{equation}
\rho^{4}-m\rho^{2}+2(m-3)=0. \label{rho}%
\end{equation}
As we all know that
\[
\rho>\sqrt{\Delta(S(2,n^{\prime}-4))}=\sqrt{n^{\prime}-3}=\sqrt{m-2},
\]
where $\Delta(S(2,n^{\prime}-4))$ is the maximum degree of the tree
$S(2,n^{\prime}-4).$

Take $\alpha=\frac{1}{\rho^{2}}.$ Then $\alpha<\frac{1}{m-2}$ and
\[
1-\alpha\geq1-(m-3)\alpha>1-\frac{m-3}{m-2}>0.
\]
So $B(v,e)\ >0$ for any vertex $v$ and any edge $e$ of $\mathcal{T(}%
1,1,m-3\mathcal{)}$ when $v\in e,$ i.e., the matrix $B$ is a weighted
incidence matrix of $\mathcal{T(}1,1,m-3\mathcal{)}$ according to Definition
\ref{weighted incidence matrix}. Now we will show $\mathcal{T(}%
1,1,m-3\mathcal{)}$ is strictly $\alpha$-subnormal with $\alpha=\frac{1}%
{\rho^{2}}$. Combining (\ref{sum}) and (\ref{P1}), we only need to show $%
%TCIMACRO{\dprod \limits_{v:v\in e_{0}}}%
%BeginExpansion
{\displaystyle\prod\limits_{v:v\in e_{0}}}
%EndExpansion
B(v,e_{0})>\alpha.$ In fact by (\ref{alapha}) and (\ref{rho}) we have
\begin{align*}%
%TCIMACRO{\dprod \limits_{v:v\in e_{0}}}%
%BeginExpansion
{\displaystyle\prod\limits_{v:v\in e_{0}}}
%EndExpansion
B(v,e_{0})-\alpha &  =-(m-3)\alpha^{3}+(2m-5)\alpha^{2}-m\alpha+1\\
&  =\frac{1}{\rho^{6}}[\rho^{6}-m\rho^{4}+(2m-5)\rho^{2}-(m-3)]\\
&  =\frac{1}{\rho^{6}}[\rho^{2}-(m-3)]\\
&  >0.
\end{align*}
So for the unique non-pendent edge $e_{0}$ we have
\begin{equation}%
%TCIMACRO{\dprod \limits_{v:v\in e_{0}}}%
%BeginExpansion
{\displaystyle\prod\limits_{v:v\in e_{0}}}
%EndExpansion
B(v,e_{0})>\alpha. \label{P2}%
\end{equation}

By (1) of Theorem \ref{LuMan}, we have
\[
\rho(\mathcal{T(}1,1,m-3\mathcal{)})<\alpha^{-\frac{1}{k}}=\rho^{\frac{2}{k}%
}=\rho(S^{k}(2,n^{\prime}-4)).
\]
(2). Write $\rho=\rho(F_{n^{\prime}})$ for short. By Theorem \ref{Bu}, we have
$\rho(F_{n^{\prime}}^{k})=\rho^{\frac{2}{k}}.$ It is easy to see that the tree
$F_{n^{\prime}}$ contains $m$ edges and the value $\rho$ satisfies
\begin{equation}
\rho^{4}-(m-1)\rho^{2}+(m-4)=0, \label{rho2}%
\end{equation}
and
\[
\rho>\sqrt{\Delta(F_{n^{\prime}})}=\sqrt{n^{\prime}-3}=\sqrt{m-2},
\]
where $\Delta(F_{n^{\prime}})$ is the maximum degree of the tree
$F_{n^{\prime}}.$

Take $\alpha=\frac{1}{\rho^{2}}.$ Then $\alpha<\frac{1}{m-2}$ and
\[
1-\alpha\geq1-(m-3)\alpha>1-\frac{m-3}{m-2}>0.
\]
So $B(v,e)\ >0$ for any \ vertex $v$ and any edge $e$ of $\mathcal{T(}%
1,1,m-3\mathcal{)}$ when $v\in e,$ i.e., the matrix $B$ is a weighted
incidence matrix of the supertree $\mathcal{T(}1,1,m-3\mathcal{)}$. Now we
will show $\mathcal{T(}1,1,m-3\mathcal{)}$ is strictly $\alpha$-supernormal
with $\alpha=\frac{1}{\rho^{2}}.$ Combining (\ref{sum}) and (\ref{P1}), we
only need to show $%
%TCIMACRO{\dprod \limits_{v:v\in e_{0}}}%
%BeginExpansion
{\displaystyle\prod\limits_{v:v\in e_{0}}}
%EndExpansion
B(v,e_{0})<\alpha.$ In fact by (\ref{alapha}) and (\ref{rho2}) we have
\begin{align*}%
%TCIMACRO{\dprod \limits_{v:v\in e_{0}}}%
%BeginExpansion
{\displaystyle\prod\limits_{v:v\in e_{0}}}
%EndExpansion
B(v,e_{0})-\alpha &  =-(m-3)\alpha^{3}+(2m-5)\alpha^{2}-m\alpha+1\\
&  =\frac{1}{\rho^{6}}[\rho^{6}-m\rho^{4}+(2m-5)\rho^{2}-(m-3)]\\
&  =\frac{1}{\rho^{6}}[-\rho^{4}+(m-1)\rho^{2}-(m-3)]\\
&  =-\frac{1}{\rho^{6}}\\
&  <0.
\end{align*}
So for the unique non-pendent edge $e_{0}$ we have
\begin{equation}%
%TCIMACRO{\dprod \limits_{v:v\in e_{0}}}%
%BeginExpansion
{\displaystyle\prod\limits_{v:v\in e_{0}}}
%EndExpansion
B(v,e_{0})<\alpha. \label{P3}%
\end{equation}
Clearly, the weighted incidence matrix $B$ of $\mathcal{T(}1,1,m-3\mathcal{)}$
is consistent, since the supertree $\mathcal{T(}1,1,m-3\mathcal{)}$ is
acyclic. By (2) of Theorem \ref{LuMan}, we have
\[
\rho(\mathcal{T(}1,1,m-3\mathcal{)})>\alpha^{-\frac{1}{k}}=\rho^{\frac{2}{k}%
}=\rho(F_{n^{\prime}}^{k}).
\]
The proof is complete.
\end{proof}

\section{The proofs of the main results}

Suppose that $n^{\prime}=\frac{n-1}{k-1}+1,$ and $m=n^{\prime}-1.$ Recall that
$N_{2}(\mathcal{T})$ is the number of non-pendent vertices of a supertree
$\mathcal{T}.$ For a $k$-uinform supertree $\mathcal{T}$ on $n$ vertices we
have the following observations.

(1). $N_{2}(\mathcal{T})=1$ if and only if $\mathcal{T\cong}S_{n^{\prime}}%
^{k};$

(2). $N_{2}(\mathcal{T})=2$ if and only if $\mathcal{T\cong}S^{k}(a,b)$ for
some integers $a,b$, where $b\geq a\geq1$ and $a+b=n^{\prime}-2;$

(3.1). $N_{2}(\mathcal{T})=3$ and three non-pendent vertices incident to one
edge if and only if $\mathcal{T\cong T(}t_{1},t_{2},t_{3}\mathcal{)}$ for some
interges $t_{1},t_{2},t_{3}$, where $t_{1}+t_{2}+t_{3}=m-1.$

(3.2). $N_{2}(\mathcal{T})=3$ and three non-pendent vertices not incident to
one edge, if and only if $\mathcal{T\cong}T^{k}$ for some ordinary tree $T$
and $T$ containing three non-pendent vertices.

\bigskip

\begin{proof}
[Proof of Theorem \ref{Main 1}]Since $\mathcal{T\ncong}S_{n^{\prime}}^{k},$ we
have $N_{2}(\mathcal{T})\geq2.$

If $N_{2}(\mathcal{T})=2,$ then $\mathcal{T\cong}S^{k}(a,b)$ for some integers
$a,b$, where $b\geq a\geq1$ and $a+b=n^{\prime}-2.$ Since $\mathcal{T\ncong
}S^{k}(1,n^{\prime}-3),$ by Lemma \ref{S(a,b)}, we have
\[
\rho(\mathcal{T})\leq\rho(S^{k}(2,n^{\prime}-4)),
\]
with equality holding if and if $\mathcal{T\cong}S^{k}(2,n^{\prime}-4).$

If $N_{2}(\mathcal{T})=3$ and $\mathcal{T\cong T(}t_{1},t_{2},t_{3}%
\mathcal{)},$ then combining Lemmas \ref{T(a,b,c)} and \ref{T(1,,1,)} we have
\[
\rho(\mathcal{T})\leq\rho(\mathcal{T(}1,1,m-3\mathcal{)})<\rho(S^{k}%
(2,n^{\prime}-4)).
\]

If $N_{2}(\mathcal{T})=3$ and $\mathcal{T\cong}T^{k}$ for some ordinary tree
$T,$ then $T$ contains three non-pendent vertices and then $T\notin
\{S_{n^{\prime}},S(a,b)\}.$ From Corollary \ref{Ordering Tk}, we have
\[
\rho(\mathcal{T})<\rho(S^{k}(2,n^{\prime}-4)).
\]

If $N_{2}(\mathcal{T})\geq4,$ then there exists a $k$-uinform supertree
$\mathcal{T}^{\prime}$ with $N_{2}(\mathcal{T}^{\prime})=3$ and $\rho
(\mathcal{T}^{\prime})>\rho(\mathcal{T})$ by Lemma \ref{induction}. Thus we
have
\[
\rho(\mathcal{T})<\rho(\mathcal{T}^{\prime})<\rho(S^{k}(2,n^{\prime}-4)).
\]

The proof is complete.
\end{proof}

\bigskip

\begin{proof}
[Proof of Theorem \ref{Main 2}]Since $\mathcal{T\ncong}S_{n^{\prime}}^{k},$ we
have $N_{2}(\mathcal{T})\geq2.$

If $N_{2}(\mathcal{T})=2,$ then $\mathcal{T\cong}S^{k}(a,b)$ for some interges
$a,b$, where $b\geq a\geq1,$ and $a+b=n^{\prime}-2.$ Since Since
$\mathcal{T\notin}\{S^{k}(1,n^{\prime}-3),S^{k}(2,n^{\prime}-4)\},$ by Lemma
\ref{S(a,b)}, Corollary \ref{Ordering Tk} and Lemma \ref{T(1,,1,)}, we have
\[
\rho(\mathcal{T})\leq\rho(S^{k}(3,n^{\prime}-5))<\rho(F_{n^{\prime}}^{k}%
)<\rho(\mathcal{T(}1,1,m-3\mathcal{)}).
\]

If $N_{2}(\mathcal{T})=3$ and $\mathcal{T\cong T(}t_{1},t_{2},t_{3}%
\mathcal{)},$ then from Lemma \ref{T(a,b,c)} we have
\[
\rho(\mathcal{T})\leq\rho(\mathcal{T(}1,1,m-3\mathcal{)}),
\]
with equality holding if and only if $\mathcal{T\cong T(}1,1,m-3).$

If $N_{2}(\mathcal{T})=3$ and $\mathcal{T\cong}T^{k}$, then $T\notin
\{S_{n^{\prime}},S(a,b)\}.$ From Corollary \ref{Ordering Tk}, Lemma
\ref{T(1,,1,)} we have
\[
\rho(\mathcal{T})\leq\rho(F_{n^{\prime}}^{k})<\rho(\mathcal{T(}%
1,1,m-3\mathcal{)}).
\]

If $N_{2}(\mathcal{T})\geq4,$ then there exists a $k$-uinform supertree
$\mathcal{T}^{\prime}$ with $N_{2}(\mathcal{T}^{\prime})=3$ and $\rho
(\mathcal{T}^{\prime})>\rho(\mathcal{T})$ by Lemma \ref{induction}$.$ Thus we
have
\[
\rho(\mathcal{T})<\rho(\mathcal{T}^{\prime})\leq\rho(\mathcal{T(}%
1,1,m-3\mathcal{)}).
\]

The proof is complete.
\end{proof}

\end{document}